\documentclass[preprint,12pt]{elsarticle}

\usepackage{amssymb}
\usepackage{amsthm}
\usepackage{amsmath}
\usepackage{fixltx2e}
\usepackage[english]{babel}
\usepackage{hyphenat}

\theoremstyle{plain}
\newtheorem{theorem}{Theorem}
\newtheorem{lemma}{Lemma}

\newtheorem{proposition}{Proposition}
\theoremstyle{definition}
\newtheorem{definition}{Definition}

\theoremstyle{remark}
\newtheorem{remark}{Remark}

\newcommand{\R}{\mathbb R}
\newcommand{\N}{\mathbb N}
\newcommand{\Z}{\mathbb Z}
\newcommand{\abs}[1]{\left| #1 \right|}
\newcommand{\cf}{\operatorname{cf}}

\newcommand{\thmpart}[1]{#1.}
\newcommand{\itemref}[1]{\ref{item:#1}}
\newcommand{\itemfollows}[2]{\itemref{#1} $\Rightarrow$ \itemref{#2}}

\journal{Topology and its Applications}

\begin{document}

\begin{frontmatter}

\title{Urysohn's metrization theorem for higher cardinals}

\author{Joonas Ilmavirta}

\address{Department of Physics,
University of Jyv\"askyl\"a,\newline
P.O. Box 35 (YFL)
FI-40014 University of Jyv\"askyl\"a,
Finland}
\ead{joonas.ilmavirta@jyu.fi}

\begin{abstract}
In this paper a generalization of Urysohn's metrization theorem is given for higher cardinals. Namely, it is shown that a topological space with a basis of cardinality at most $|\omega_\mu|$ or smaller is $\omega_\mu$-metrizable if and only if it is $\omega_\mu$-additive and regular, or, equivalently, $\omega_\mu$-additive, zero-dimensional, and T\textsubscript{0}.
Furthermore, all such spaces are shown to be embeddable in a suitable generalization of Hilbert's cube.

\end{abstract}


\end{frontmatter}


\section{Introduction}
\label{sec:intro}

In this section the concept of $\omega_\mu$\hyp{}metric spaces is defined and briefly discussed. For a more elaborate description of $\omega_\mu$\hyp{}metric spaces, see e.g.~\cite{top-char}. Section~\ref{sec:prelim} is dedicated to some preliminary results which are then used to prove an extension of Urysohn's metrization theorem 
in section~\ref{sec:metr}.

A number of properties of a topological space equivalent with $\omega_\mu$\hyp{}metrizability were given
by Hodel~\cite{metr-thms}
and by Nyikos and Reichel~\cite{top-char}.
The special case of spaces with a basis of cardinality at most $\abs{\omega_\mu}$ seems not to have been considered. A complete characterization of such $\omega_\mu$\hyp{}metrizable topological spaces is exhibited here in Theorem~\ref{thm:uryson} in terms of simple topological properties which may be easier to verify than those required by more general metrization theorems.

If $G$ is an ordered abelian group, $X$ is a nonempty set and $d\!:\!X \times X \to G$ is a function such that
\begin{enumerate}
	\item $d(x,y) \geq 0$ for all $x,y \in X$,
	\item $d(x,y) = 0$ only if $x=y$,
	\item $d(x,y) = d(y,x)$ for all $x,y \in X$, and
	\item $d(x,y) \leq d(x,z) + d(z,y)$ for all $x,y,z \in X$,
\end{enumerate}
then $d$ is a \emph{metric of $X$ over $G$}. Any such metric $d$ gives rise to a topology of $X$ whose basis consists of the open balls
\begin{equation}
	B_d(x,r) = \left\{y \in X : d(x,y) < r\right\},
\end{equation}
for all $x \in X$ and $r \in G$, $r > 0$. Where no confusion is possible, the subscript $d$ is omitted. If $X$ is a topological space and there is a metric $d$ of $X$ over $G$ giving rise to the same topology, then $X$ is \emph{$G$\hyp{}metrizable} and the pair $(X,d)$ is a \emph{$G$\hyp{}metric space}.

Of special interest are the groups $\Z^\alpha$ and $\R^\alpha$, where $\alpha$ is any non-zero ordinal, with lexicographical order and componentwise addition. Any element $x\in\Z^\alpha$ is a sequence $(x_\lambda)_{\lambda\in\alpha}$ indexed by $\alpha$ with each $x_\lambda\in\Z$, and similarly for $\R^\alpha$.

For any $\lambda\in\alpha$,
we define elements $r^\lambda$ in $\Z^\alpha$ (or in $\R^\alpha$)
by setting $r^\lambda_\lambda=1$ and $r^\lambda_\nu=0$ for all $\nu\neq\lambda$. One can immediately see that if $\alpha$ is a regular ordinal and $(X,d)$ is a $Z^\alpha$\hyp{}metric (or $R^\alpha$\hyp{}metric) space, then the collection
\begin{equation}
	\left\{B_d(x,r^\lambda) : x \in X, \lambda\in\alpha\right\}
\end{equation}
is a basis for the topology of $X$.

\begin{proposition}
\label{prop:cof reduction}
If a topological space $X$ is $\Z^\alpha$\hyp{}metrizable, it is also $\Z^{\cf\alpha}$\hyp{}metrizable, where $\cf\alpha$ is the cofinality of $\alpha$. Similarly, if $X$ is $\R^\alpha$\hyp{}metrizable, it is also $\R^{\cf\alpha}$\hyp{}metrizable.
\end{proposition}
\begin{proof}
Let $X$ be a $\Z^\alpha$\hyp{}metric space, and let $L \subset \alpha$ be a cofinal subset order isomorphic to $\cf\alpha$. Then the topology given by the basis
\begin{equation}
	\left\{B(x,r^\lambda) : x \in X, \lambda\in L\right\}
\end{equation}
is immediately seen to be identical to the original metric topology of $X$. Thus $X$ is $\Z^L$\hyp{}metrizable. The same argument holds for $\R^\alpha$\hyp{}metrizable spaces.
\end{proof}

\begin{proposition}
\label{prop:Z R}
Let $\alpha$ be an infinite regular ordinal. A topological space $X$ is $\Z^\alpha$\hyp{}metrizable if and only if it is $\R^\alpha$\hyp{}metrizable.
\end{proposition}
\begin{proof}
Since $\Z^\alpha$ is a subgroup of $\R^\alpha$, any $\Z^\alpha$\hyp{}metric space is trivially an $\R^\alpha$\hyp{}metric space.

Let then $(X,d)$ be a $\R^\alpha$\hyp{}metric space. For any points $x,y \in X$, $x \neq y$, let $n_{xy} = \min \{\lambda\in\alpha:d_\lambda(x,y) \neq 0\}$. One can define a $\Z^\alpha$\hyp{}metric $\delta$ for $X$ by setting $\delta(x,x)=0$ and $\delta(x,y)=r^{n_{xy}}$ when $x \neq y$.

Let $x,y,z \in X$ be any three distinct points. Because $d$ obeys the triangle inequality, one has $\min\{n_{xz},n_{yz}\} \leq n_{xy}$, and so $\max\{r^{n_{xz}},r^{n_{yz}}\} \geq r^{n_{xy}}$. This implies $\max\{\delta(x,z),\delta(y,z)\} \geq \delta(x,y)$, and hence $\delta(x,z)+\delta(y,z) \geq \delta(x,y)$.
Thus $\delta$ obeys the triangle inequality; the other conditions for a metric are obviously fulfilled. Since $B_\delta(x,r^\lambda) \supset B_d(x,r^{\lambda+1})$ and $B_d(x,r^\lambda) \supset B_\delta(x,r^{\lambda+1})$ for all $x \in X$ and $\lambda\in\alpha$, the two metric topologies are the same.
\end{proof}

Due to Proposition~\ref{prop:cof reduction} only regular ordinals $\alpha$ need to be considered. Every infinite regular ordinal $\alpha$ is an initial ordinal, that is $\alpha=\omega_\mu$ for some ordinal $\mu$. Moreover, for a finite $\alpha$ one has $\cf\alpha=1$, which yields either the discrete $\Z^1$\hyp{}metric or the usual $\R^1$\hyp{}metric.

\begin{definition}
\label{def:omega metr}
A topological space $X$ is an \emph{$\omega_\mu$\hyp{}metric space} if it is a $\Z^{\omega_\mu}$\hyp{}metric space (equivalently $\R^{\omega_\mu}$\hyp{}metric space by Proposition~\ref{prop:Z R}).
\end{definition}

Unless otherwise stipulated, every $\omega_\mu$\hyp{}metric will be assumed to take values in $\Z^{\omega_\mu}$ instead of $\R^{\omega_\mu}$.


\section{Preliminaries}
\label{sec:prelim}


Let $\kappa$ be a cardinal. A topological space $X$ is \emph{$\kappa$\hyp{}additive} if for any collection
$\{U_i:i \in I\}$ 
of open sets in $X$ the intersection $\bigcap_{i \in I}U_i$ is open whenever $\abs{I}<\kappa$. An $\abs{\omega_\mu}$\hyp{}additive space is also called an $\omega_\mu$\hyp{}additive space.

\begin{proposition}
\label{prop:N2 sep lind top}
Let $X$ be topological space with a basis of cardinality $\kappa$ or smaller. Then
\begin{enumerate}
	\item $X$ contains a dense set whose cardinality is at most $\kappa$, and
	\item every open cover of $X$ has a subcover whose cardinality is at most $\kappa$.
\end{enumerate}
\end{proposition}
\begin{proof}
Let $\mathcal{B}$ be a basis for the topology of $X$ such that $\abs{\mathcal{B}}\leq\kappa$.

\thmpart{1} For every set $A \in \mathcal{B}$ there is an element $x_A \in A$. The set $\{x_A:A \in \mathcal{B}\}$ is obviously dense and its cardinality cannot exceed $\kappa$.

\thmpart{2} Let $\mathcal{C}$ be any open cover of $X$. If $\mathcal{B}'$ is the collection of sets $B\in\mathcal{B}$ such that $B \subset U$ for some $U\in\mathcal{C}$, one can choose for every $B\in\mathcal{B}'$ a set $U_B$ with $B \subset U_B \in \mathcal{C}$. The collection $\mathcal{C}'=\{U_B:B \in \mathcal{B}'\}$ is the subcover sought for: indeed, each $x \in X$ is an interior point of some $U \in \mathcal{C}$, and so $x \in B \subset U$ for some $B\in\mathcal{B}$ and thus $\bigcup_{B\in\mathcal{B}'}U_B=X$.
\end{proof}

\begin{proposition}
\label{prop:N2 sep lind metr}
In any $\omega_\mu$\hyp{}metric space $X$ either one of the following two conditions is sufficient to guarantee that the topology of $X$ have a basis of cardinality $\abs{\omega_\mu}$ or smaller:
\begin{enumerate}
	\item $X$ contains a dense subset whose cardinality is at most $\abs{\omega_\mu}$.
	\item Every open cover of $X$ has a subcover whose cardinality is at most $\abs{\omega_\mu}$.
\end{enumerate}
\end{proposition}
\begin{proof}
\thmpart{1} Let $A \subset X$ be a dense subset such that $\abs{A}\leq\abs{\omega_\mu}$.
The collection $\mathcal{B}=\{B(a,r^\lambda):a \in A, \lambda\in\omega_\mu\}$ has obviously cardinality $\abs{\omega_\mu}$ or smaller.
It remains to show that $\mathcal{B}$ is also a basis for the topology. Let $U \subset X$ be any nonempty open set and let $x \in U$. Then there is $\lambda\in\omega_\mu$ such that $B(x,r^\lambda) \subset U$, and one can find a point $a \in A \cap B(x,r^{\lambda+1})$. Now $x \in B(a,r^{\lambda+1}) \subset B(x,r^{\lambda}) \subset U$ and $B(a,r^{\lambda+1}) \in \mathcal{B}$. Hence $\mathcal{B}$ indeed is a basis.

\thmpart{2} For every $\lambda\in\omega_\mu$ the collection $\{B(x,r^\lambda):x \in X\}$ is an open cover of $X$. Therefore there is a set $A_\lambda \subset X$ such that $\abs{A_\lambda}\leq\abs{\omega_\mu}$ and $\{B(x,r^\lambda):x \in A_\lambda\}$ is an open cover of $X$. The union $A=\bigcup_{\lambda\in\omega_\mu}A_\lambda$ is dense in $X$ and has cardinality $\abs{A}\leq\abs{\omega_\mu}\times\abs{\omega_\mu}=\abs{\omega_\mu}$.
\end{proof}

\begin{lemma}
\label{lemma:T3 lind T4}
Let $X$ be a T\textsubscript{3}-space and let $\kappa$ be a cardinal.
Assume that $X$ is $\kappa$\hyp{}additive and every open cover of $X$ has a subcover of cardinality $\kappa$ or smaller.
Then $X$ is a T\textsubscript{4}-space.
\end{lemma}
\begin{proof}
Let $E$ and $F$ be disjoint nonempty closed sets in $X$. Since $X$ is T\textsubscript{3}, for every $e \in E$ one can find a neighborhood $U_e$ such that
$F \subset X\setminus\overline{U}_e$.
Similarly every $f \in F$ has a neighborhood $V_f$ with
$E \subset X\setminus\overline{V}_f$.
Since
\begin{equation}
\mathcal{C} = \{X \setminus (E \cup F)\} \cup \{U_e:e \in E\} \cup \{V_f:f \in F\}
\end{equation}
is an open cover of $X$, the assumed covering property guarantees that $\mathcal{C}$ has a subcover
\begin{equation}
\mathcal{C}' = \{X \setminus (E \cup F)\} \cup \{U_{e_\lambda}:\lambda\in\alpha\} \cup \{V_{f_\lambda}:\lambda\in\alpha\},
\end{equation}
where each $e_\lambda \in E$ and $f_\lambda \in F$, and where $\alpha$ is the initial ordinal of the cardinal $\kappa$.

For every $\lambda\in\alpha$ the sets
\begin{equation}
\begin{split}
A_\lambda &= U_{e_\lambda}\setminus\bigcup_{\nu<\lambda}\overline{V}_{f_\nu} \quad\text{and}\\
B_\lambda &= V_{f_\lambda}\setminus\bigcup_{\nu<\lambda}\overline{U}_{e_\nu}
\end{split}
\end{equation}
are open by hypothesis, and hence the sets
\begin{equation}
A=\bigcup_{\lambda\in\alpha}A_\lambda \quad\text{and}\quad B=\bigcup_{\lambda\in\alpha}B_\lambda
\end{equation}
are neighborhoods of $E$ and $F$, respectively. These neighborhoods are disjoint;
for if $\nu\leq\lambda$, then $\overline{B}_\nu \subset X \setminus A_\lambda$, and so $A_\lambda \cap B_\nu =\emptyset$ and similarly in the case $\nu\geq\lambda$.
\end{proof}

The following two known lemmas are elementary, and they are included only for the sake of an easy reference.

\begin{lemma}
\label{lemma:0-dim}
If a topological space $X$ is zero-dimensional and T\textsubscript{0}, then it is also T\textsubscript{2} and T\textsubscript{3}.
\end{lemma}

\begin{lemma}
\label{lemma:nested basis}
Let $X$ be a topological T\textsubscript{3}-space with a basis $\mathcal{B}$ and let $x \in X$. Then for each neighborhood $U$ of $x$ there are $B,B'\in\mathcal{B}$ such that $x \in B \subset \overline{B} \subset B' \subset U$.
\end{lemma}

\section{The metrization theorem}
\label{sec:metr}

In order to extend Urysohn's metrization theorem to higher cardinals and thus to $\omega_\mu$\hyp{}metric spaces, a generalization of Hilbert's cube is needed.
The product topology of $\{0,1\}^{\omega_\mu}$ is not suitable for this purpose, since it is not $\omega_\mu$\hyp{}additive for $\mu>0$.

\begin{definition}
Let $\Z^{\omega_\mu}$ be given an $\omega_\mu$\hyp{}metric $d$ by defining $d(x,y)_\lambda=\abs{x_\lambda-y_\lambda}$ for every $\lambda\in\omega_\mu$. The set $Q_\mu=\{0,1\}^{\omega_\mu}\subset\Z^{\omega_\mu}$ with the $\omega_\mu$\hyp{}metric inherited from $\Z^{\omega_\mu}$ is \emph{the generalized Hilbert's cube}.
\end{definition}

A basis for the topology of the cube $Q_\mu$ consists of the products $\prod_{\lambda\in\omega_\mu}U_\lambda$,
where there is $\nu\in\omega_\mu$ such that $U_\lambda$ is a singleton when $\lambda<\nu$ and $U_\lambda=\{0,1\}$ when $\lambda\geq\nu$. The cardinality of this basis is $\abs{\omega_\mu}$.

The embedding theorem~\ref{thm:embed2}, which will be stated and proven shortly, will make use of the classical Urysohn's lemma:

\begin{lemma}[Urysohn's lemma]
\label{lemma:uryson}
Let $X$ be a T\textsubscript{4}-space and let $E$ and $F$ be disjoint nonempty closed sets in $X$. Then there is a continuous mapping $f\!:\!X \to [0,1]$ which satisfies $f(E)=\{0\}$ and $f(F)=\{1\}$.
\end{lemma}

\begin{lemma}
\label{lemma:uryson2}
Let $X$ be a T\textsubscript{4}-space and let $E$ and $F$ be disjoint nonempty closed sets in $X$. If $X$ is $\omega_1$\hyp{}additive, there is a continuous mapping $f\!:\!X \to \{0,1\}$ which satisfies $f(E)=\{0\}$ and $f(F)=\{1\}$.
\end{lemma}
\begin{proof}
Let $g\!:\!X \to [0,1]$ be a mapping given by Urysohn's lemma. Define $f\!:\!X \to \{0,1\}$ by setting $f(x)=0$ when $g(x)=0$ and $f(x)=1$ otherwise.
Every set $g^{-1}([0,1/n))$, $n\in\N$, is open in $X$ and therefore, by hypothesis, so is the intersection $f^{-1}(\{0\})=g^{-1}(\{0\})=\bigcap_{n\in\N}g^{-1}([0,1/n))$. Thus $f$ is continuous.
\end{proof}

\begin{theorem}
\label{thm:embed2}
Let $X$ be a topological T\textsubscript{1}- and T\textsubscript{3}-space.
If $X$ is $\omega_\mu$\hyp{}additive and has a basis of cardinality $\abs{\omega_\mu}$ or smaller
for
a regular ordinal $\omega_\mu$, $\mu>0$,
then $X$ can be embedded in the generalized Hilbert's cube $Q_\mu$.
\end{theorem}
\begin{proof}
It follows from Proposition~\ref{prop:N2 sep lind top} and Lemma~\ref{lemma:T3 lind T4} that $X$ is also T\textsubscript{4}.

Let $\{B_j \subset X : j \in J\}$ be a basis for $X$ such that $\abs{J}\leq\abs{\omega_\mu}$.
Let $P$ be the set of pairs $(i,j) \in J \times J$ for which $\overline{B_i} \subset B_j$. Since $\abs{P} \leq \abs{J}$, the elements of $P$ can be indexed so that $P=\{(i_\lambda,j_\lambda):\lambda\in\omega_\mu\}$.

For every $\lambda\in\omega_\mu$ a continuous mapping $f_\lambda\!:\!X \to \{0,1\}$ is chosen such that $f_\lambda(\overline{B_{i_\lambda}})=\{1\}$ and $f_\lambda(X\setminus B_{j_\lambda})=\{0\}$. This is possible by Lemma~\ref{lemma:uryson2}.
We define $f\!:\!X \to Q_\mu = \{0,1\}^{\omega_\mu}$ componentwise by the mappings $f_\lambda$ and show that $f$ embeds $X$ in $Q_\mu$.

The mapping $f$ is continuous: Let $x \in X$ be a point and let $U$ be a neighborhood of $f(x)$. Then there exist $\nu\in\omega_\mu$ and sets $U_\lambda \subset \{0,1\}$ with $f_\lambda(x)\in U_\lambda$ for all $\lambda\in\omega_\mu$ and $U_\lambda=\{0,1\}$ when $\lambda\geq\nu$, so that $U$ contains the product $\prod_{\lambda\in\omega_\mu}U_\lambda$.
Because each $f_\lambda$ is continuous,
for every $\lambda<\nu$
one can find an open set $V_\lambda \subset X$ so that $f_\lambda(V_\lambda) \subset U_\lambda$.
Since $X$ is $\omega_\mu$\hyp{}additive, the set $V=\bigcap_{\lambda<\nu}V_\lambda$ is a neighborhood of $x$. Obviously $f(V) \subset \prod_{\lambda\in\omega_\mu}U_\lambda \subset U$.

The mapping $f$ is one-to-one: Let $x,y \in X$ be two distinct points.
By the T\textsubscript{1}-property and Lemma~\ref{lemma:nested basis} there are $i,j \in J$ such that $x \in B_i \subset \overline{B_i} \subset B_j$.
Thus $(i,j)=(i_\lambda,j_\lambda)$ for some $\lambda\in\omega_\mu$. Now $f_\lambda(x)=1$ and $f_\lambda(y)=0$, and so $f(x) \neq f(y)$.

Let $g\!:\!f(X) \to X$ be the inverse of $f$. It remains to show that $g$ is continuous.
Fix $x \in X$ and let $U$ be a neighborhood of $x$.
By Lemma~\ref{lemma:nested basis} there is $\lambda\in\omega_\mu$ for which $x \in B_{i_\lambda}$ and $B_{j_\lambda} \subset U$.
Now $f_\lambda(x)=1$ and the set $V=\prod_{\nu\in\omega_\mu}V_\nu$, where $V_\lambda=\{1\}$ and $V_\nu=\{0,1\}$ for $\nu\neq\lambda$, is a neighborhood of $f(x)$ in $Q_\mu$.
For every $y \in g(V \cap f(X))$ we have $f(y) \in V$, $f_\lambda(y)=1$, and so $y \in B_{j_\lambda} \subset U$. Thus $g(V \cap f(X)) \subset U$ and $g$ is continuous.
\end{proof}

\begin{theorem}
\label{thm:uryson}
For any topological space $X$ and any regular ordinal $\omega_\mu>\omega_0$ the following are equivalent:
\begin{enumerate}
	\item\label{item:1} $X$ is $\omega_\mu$\hyp{}additive, T\textsubscript{0}, zero-dimensional and has a basis of cardinality $\abs{\omega_\mu}$ or smaller.
	\item\label{item:4} $X$ is $\omega_\mu$\hyp{}additive, T\textsubscript{1} and T\textsubscript{3}, and has a basis of cardinality $\abs{\omega_\mu}$ or smaller.
	\item\label{item:2} $X$ is $\omega_\mu$\hyp{}metrizable and has a basis of cardinality $\abs{\omega_\mu}$ or smaller.
	\item\label{item:2b} $X$ is $\omega_\mu$\hyp{}metrizable and contains a dense set of cardinality $\abs{\omega_\mu}$ or smaller.
	\item\label{item:2c} $X$ is $\omega_\mu$\hyp{}metrizable and every open cover of $X$ has a subcover of cardinality $\abs{\omega_\mu}$ or smaller.
	\item\label{item:3} $X$ can be embedded in $Q_\mu$.
\end{enumerate}
\end{theorem}

\begin{proof}
Lemma~\ref{lemma:0-dim} and Theorem~\ref{thm:embed2} provide the implications \itemfollows{1}{4} and \itemfollows{4}{3}.
The generalized Hilbert's cube $Q_\mu$ is $\omega_\mu$\hyp{}metrizable by definition and has a basis of cardinality $\abs{\omega_\mu}$, whence \itemfollows{3}{2}.
By Propositions~\ref{prop:N2 sep lind top} and~\ref{prop:N2 sep lind metr} the conditions \itemref{2}, \itemref{2b}, and \itemref{2c} are equivalent.

The implication \itemfollows{2b}{1} is seen as follows. Let $A$ be a dense subset of $X$ such that $\abs{A}\leq\abs{\omega_\mu}$. The collection $\{B(a,r^\lambda):a \in A,\lambda\in\omega_\mu\}$ can easily be verified to be a clopen basis for $X$, and its cardinality is manifestly $\abs{\omega_\mu}$ or smaller.
\end{proof}

\begin{remark}
The set $Q_\mu$ is considered to be a generalization of Hilbert's cube due to its role in Theorem~\ref{thm:uryson}. The cube $Q_0$, however, is a Cantor set with its usual topology, and so the theorem does not hold for $\mu=0$.
\end{remark}

\begin{remark}
Let $\omega_\mu$, $\mu>0$, be a regular ordinal. Consider a topological space $X$ which has two bases, one of cardinality $\abs{\omega_\mu}$ or smaller and the other consisting of clopen sets. Does $X$ have a basis which has both of these properties?

First, $X$ can be assumed to be a T\textsubscript{0}-space; it is sufficient to consider the bases for the Kolmogorov quotient $KQ(X)$ of $X$, and $KQ(X)$ is a T\textsubscript{0}-space.
By Theorem~\ref{thm:uryson} $\omega_\mu$\hyp{}additivity is sufficient for such a basis to exist.
If $X$ is strongly zero-dimensional, every open cover of $X$ has a refinement where the covering sets are disjoint. Any such refinement of the basis for $X$ that has cardinality $\abs{\omega_\mu}$ or smaller is a suitable clopen basis.
Without any further assumptions, however, it is unclear whether or not such a basis exists.
\end{remark}

\section{Acknowledgements}

The author wishes to acknowledge the suggestions offered to him by conversations with Heikki Junnila.


\bibliographystyle{elsarticle-num}
\bibliography{uryson}

\end{document}